\documentclass[10pt, a4paper]{amsart}
\usepackage{amsmath}
\usepackage{amssymb}
\usepackage{amsthm}
\usepackage{graphicx}
\usepackage{url}
\usepackage{enumerate}
\usepackage{xcolor}
\usepackage{mathrsfs} 

\newtheorem{theorem}{Theorem}
\newtheorem{lemma}[theorem]{Lemma}
\newtheorem{rem}[theorem]{Remark}
\newtheorem{proposition}[theorem]{Proposition}
\newtheorem{corollary}[theorem]{Corollary}

\theoremstyle{definition}

\DeclareMathOperator{\Int}{int}

\DeclareMathOperator{\M}{\mathcal{M}}
\DeclareMathOperator{\MT}{\mathcal{M}_T}

\newcommand{\OO}{\mathbf{0}}
\DeclareMathOperator{\F}{\mathcal{F}}
\DeclareMathOperator{\Bl}{\mathcal{B}}
\DeclareMathOperator{\lang}{\Bl}

\DeclareMathOperator{\G}{\mathcal{G}}
\DeclareMathOperator{\Mon}{\mathcal{M}}
\DeclareMathOperator{\Rst}{\mathcal{R}}

\DeclareMathOperator{\Cs}{\mathcal{C}^s}
\DeclareMathOperator{\Cp}{\mathcal{C}^p}

\DeclareMathOperator{\CsX}{\mathcal{C}^s_X}
\DeclareMathOperator{\CpX}{\mathcal{C}^p_X}

\DeclareMathOperator{\CsY}{\mathcal{C}^s_Y}
\DeclareMathOperator{\CpY}{\mathcal{C}^p_Y}

\newcommand{\V}{\mathscr{V}}
\newcommand{\U}{\mathscr{U}}
\newcommand{\alf}{\mathscr{A}}
\newcommand{\RR}{\mathbf{R}}
\newcommand{\XR}{X_\RR}

\newcommand{\htop}{h_\textrm{top}}
\newcommand{\set}[1]{\left\{#1\right\}}

\newcommand{\tword}[2]{\genfrac{\lfloor}{\rceil}{0pt}{1}{#1}{#2}}

\newcommand{\eps}{\varepsilon}

\newcommand{\Z}{\mathbb{Z}}
\newcommand{\N}{\mathbb{N}}
\newcommand{\Zp}{{\Z_+}}

\author{Dominik Kwietniak, Piotr Oprocha \and Micha{\l} Rams}

\address[D. Kwietniak]{
Faculty of Mathematics and Computer Science, Jagiellonian University in Krakow, ul. \L o\-jasiewicza 6, 30-348 Krak\'ow, Poland}\email{dominik.kwietniak@uj.edu.pl}
\urladdr{www.im.uj.edu.pl/DominikKwietniak/}

\address[P. Oprocha]{AGH University of Science and Technology, Faculty of Applied
Mathematics, al.
Mickiewicza 30, 30-059 Krak\'ow, Poland}
\email{oprocha@agh.edu.pl}

\address[M. Rams]{Institute of Mathematics, Polish Academy of Sciences,
ul. \'Sniadeckich 8, 00-956 Warszawa, Poland
}
\email[M. Rams]{rams@impan.gov.pl}

\title[On entropy of systems with almost specification]{On entropy of dynamical systems with almost specification} 
\date{\today}

\begin{document}

\begin{abstract}
We construct a family of shift spaces with almost specification and multiple measures of maximal entropy.
This answers a question from Climenhaga and Thompson [\emph{Israel J. Math.} \textbf{192} (2012), no. 2, 785--817]. Elaborating on our examples we also prove that some sufficient conditions for every subshift factor of a shift space to be intrinsically ergodic given by Climenhaga and Thompson are in some sense best possible, moreover, the weak specification property neither implies intrinsic ergodicity, nor follows from almost specification. We also construct a dynamical system with the weak specification property, which does not have the almost specification property. We prove that the minimal points are dense in the support of any invariant measure of a system with the almost specification property. Furthermore, if a system with almost specification has an invariant measure with non-trivial support, then it also has uniform positive entropy over the support of any invariant measure and can not be minimal.
\end{abstract}
\maketitle

We study dynamical systems with weaker forms of the specification property. We focus on the topological entropy and
the problem of uniqueness of a measure of maximal entropy for systems with the almost specification or weak specification
property (we also prove that these two specification-like properties are non-equivalent --- neither of them implies the other).
Recall that dynamical systems with a unique measure of maximal entropy are known as \emph{intrinsically ergodic}.
The problem of intrinsic ergodicity of shift spaces with almost specification was mentioned in \cite[p. 798]{CT}, where another approach was developed in order to prove that certain classes of symbolic systems and their factors are intrinsically ergodic. We solve the problem in the negative and provide examples of shift spaces with the weak (almost) specification property and many measures of maximal entropy\footnote{When we communicated our solution to Dan Thompson, he kindly informed us that Ronnie Pavlov had also solved the same problem (see \cite{Pavlov}). Both solutions were independently discovered in July 2014. We would like to thank Ronnie Pavlov for sharing his work with us. 
}. Our construction allows us to prove that the sufficient condition for the inheritance of intrinsic ergodicity by factors from the Climenhaga-Thompson paper \cite{CT} is optimal --- if this condition does not hold, then the symbolic systems to which Theorem of \cite{CT} applies may have a factor with many measures of maximal entropy.
We also prove that nontrivial dynamical systems with the almost specification property and a full invariant measure have uniform positive entropy and horseshoes (subsystems which are extensions of the full shift over a finite alphabet). It follows that  minimal points are dense in the measure center (the smallest closed invariant subset of the phase space which contains the support of \emph{every} invariant measure) of a system with almost specification and that these systems cannot be minimal if they are nontrivial.


\section{Basic definitions and notation}
We write $\N=\{1,2,3,\ldots\}$ and $\Zp=\{0,1,2,\ldots\}$.

A \emph{dynamical system} consists of a compact metric space $X$ together with a continuous map $T\colon X\to X$.
By $\rho$ we denote a metric on $X$ compatible with the topology. Let $\U$ and $\V$ be open covers of $X$.
By $N(\U)$ we denote the number of sets in a finite subcover of a $\U$ with smallest cardinality. By $T^{-i}\U$ ($i\in\Zp$) we mean
the cover $\set{T^{-i}(U): U\in \U}$ and $\U \vee \V=\set{U\cap V : U\in \U, V\in \V}$.
The \emph{topological entropy}
$h(T,\U)$ of an open cover $\U$ of $X$ is defined (see \cite{Walters}) as
$$
\lim_{n\to \infty}\frac{1}{n}\log N\big(\bigvee_{i=0}^{n-1}T^{-i}\U\big).
$$
The \emph{topological entropy} of $T$ is
\[
\htop(T)=\sup_{\U:\text{open cover of }X} \htop(T,\U).
\]
Let $\MT(X)$ be the space of $T$-invariant Borel probability measures on $X$. We denote the measure-theoretic 
entropy of $\mu\in\MT(X)$ by $h_\mu(T)$ (see \cite{Walters}). The variational principle states that
\[
\htop(T)=\sup_{\mu\in\MT(X)} h_\mu(T).
\]
A measure $\mu\in\MT(X)$ that attains this supremum is a \emph{measure of maximal entropy}. We say that a system $(X,T)$ is 
\emph{intrinsically ergodic} if it has a unique measure of maximal entropy.

Let $a,b\in\Zp$, $a\leq b$. The \emph{orbit segment} of $x\in X$ over $[a,b]$ is the~sequence
\[
T^{[a,b]}(x)=(T^a(x),T^{a+1}(x),\ldots, T^b(x)).
\]
We also write $T^{[a,b)}(x)=T^{[a,b-1]}(x)$.
A \emph{specification} is a family of orbit segments
\[\xi=\{T^{[a_j,b_j]}(x_j)\}_{j=1}^n\] such that  $n\in\N$ and $b_j<a_{j+1}$ for all $1\le j <n$.

The \emph{Bowen distance between $x,y\in X$ along a finite set $\Lambda\subset\N$} is
\[
\rho^T_\Lambda(x, y) = \max\{\rho(T^j(x), T^j(y)) : j \in\Lambda\}.
\]
By the \emph{Bowen ball (of radius $\eps$, centered at $x\in X$) along $\Lambda$}
we mean the set
\[
B_\Lambda(x, \eps) = \{y \in X : \rho^T_\Lambda(x, y) < \eps\}.
\]

\section{Specification and alikes}
A dynamical system has the periodic specification property if one can approximate distinct pieces of orbits by single periodic orbits with a certain uniformity.
Bowen introduced this property in \cite{Bowen} and showed that a basic set for an axiom A diffeomorphism $T$ can be partitioned into a finite number of disjoint sets $\Lambda_1,\ldots,\Lambda_k$ which are permuted by $T$ and $T^k$ restricted to $\Lambda_j$ has the specification property for each $j=1,\ldots,k$. There are many generalizations of this notion. One of them is due to Dateyama, who introduced in \cite{Dateyama} the \emph{weak specification property} (Dateyama calls it ``almost weak specification''). Dateyama's notion is a variant of a specification property used by Marcus in \cite{Marcus}   (Marcus did not coined a name for the property he stated in \cite[Lemma 2.1]{Marcus}, we think that \emph{periodic weak specification} is an appropriate name).

A  dynamical system $(X,T)$ has the \emph{weak specification property} if for every $\eps>0$ there is a function $M_\eps \colon \N \to \N$
with $\lim_{n\to \infty} M_\eps(n)/n=0$ such that for any specification $\{T^{[a_i,b_i]}(x_i)\}_{i=1}^k$ with
$a_{i} - b_{i-1}\geq M_\eps(b_i-a_i)$ for $i= 2,\dots,k$, we can find a point $x\in X$ such
that for each $i=1,\ldots, k$ and $a_i \leq j \leq b_i$, we have
\begin{equation}
\rho(T^j(x),T^j(y_i))\leq \eps \label{cond:psp1}.
\end{equation}
We say that $M_\eps$ is an $\eps$-gap function for $T$.

Marcus proved in \cite{Marcus} that the periodic point measures are weakly dense in the space of invariant measures for ergodic toral automorphisms. Dateyama
established that for an automorphism $T$ of a compact metric abelian group the weak specification property is equivalent to  ergodicity of $T$ with respect to Haar measure \cite{Dateyama2}. 

Another interesting notion is the \emph{almost specification property}. Pfister and Sullivan introduced the \emph{g-almost
product property} in \cite{PS1}. Thompson \cite{T} modified this notion slightly and renamed it the \emph{almost
specification property}. The primary examples of dynamical systems with the almost specification property are $\beta$-shifts (see \cite{CT,PS1}).
We follow Thompson's approach, hence the almost specification property presented below is
a priori weaker (less restrictive) than the notion introduced by Pfister and Sullivan.

We say that $g \colon \Zp\times(0,\eps_0)\to\N$, where $\eps_0>0$ is a \emph{mistake function} if for all $\eps<\eps_0$  and all
$n \in\Zp$  we have $g(n, \eps) \le g(n + 1, \eps)$ and
\[
\lim_{n\to \infty}
\frac{g(n, \eps)}{n}= 0.
\]
Given a mistake function $g$ we define a function $k_g\colon (0,\infty) \to \N$ by
declaring $k_g(\eps)$ to be the smallest $n\in\N$ such that $g(m,\eps)<m\eps$ for all $m\ge n$.

Given a mistake function $g$, $0<\eps<\eps_0$ and $n\ge k_g(\eps)$ 
we define the set
\[
I(g; n, \eps) := \{\Lambda\subset \set{0,1,\ldots, n - 1} : \#\Lambda \le g(n,\eps)\}.
\]

We say that a point $y\in X$ $(g;\eps,n)$-traces an orbit segment $T^{[a,b]}(x)$ over $[c,d]$ 
if $n= b-a+1=d-c+1$, $k_g(\eps)\ge n$ and for some $\Lambda\in I(g;n,\eps)$ we have
$\rho^T_\Lambda(T^a(x),T^c(y))\le\eps$. By $B_n(g;x,\eps)$ we denote the set of all points which
$(g;\eps,n)$-traces an orbit segment $T^{[0,n)}(x)$ over $[0,n)$. Note that $B_n(g;x,\eps)$ is always closed and nonempty.

A dynamical system $(X,T)$ has the \emph{almost specification
property} if there exists a mistake function $g$ such that
for any $m\geq 1$, any $\eps_1,\ldots,\eps_m > 0$,
and any specification $\{T^{[a_j,b_j]}(x_j)\}_{j=1}^m$ with $b_j-a_j+1\ge k_g(\eps_j)$ for every $j=1,\ldots,m$
we can find a point $z\in X$ which $(g;b_j-a_j+1,\eps_j)$-traces the orbit segment
$T^{[a_j,b_j]}(x_j)$ for every $j=1,\ldots,m$.

In other words,
the appropriate part of the orbit of $z$ $\eps_j$-traces with at most $g(b_j-a_j+1,\eps_j)$ mistakes the orbit of $x_j$ over $[a_j,b_j]$.


Intuitively, it should come as no surprise that almost specification does not imply the weak specification.
But we did not expect at first that the converse implication is also false.

\section{Symbolic dynamics}

We assume that the reader is familiar with the basic notions of symbolic dynamics. An excellent introduction to this theory is the book of Lind and Marcus.
We follow the notation and terminology presented there as close as possible.

Let $\Lambda$ be a finite set (an \emph{alphabet}) of \emph{symbols}. The \emph{full shift} over $\Lambda$ is the set $\Lambda^\N$ of all infinite sequences of symbols. We equip $\Lambda$ with the discrete topology and $\Lambda^\N$ with the product (Tikhonov) topology. By $\sigma$ we denote the shift operator given by $\sigma(x)_i=x_{i+1}$. A \emph{shift space} over $\Lambda$ is a closed and $\sigma$-invariant subset of $\Lambda^\N$. A \emph{block} (a \emph{word}) over $\Lambda$ is any finite sequence of symbols. The \emph{length of a block $u$}, denoted $|u|$,  is the number of symbols it contains. An \emph{$n$-block} stands for a block of length $n$. An \emph{empty block} is the unique block with no symbols and length zero. The set of all blocks over $\Lambda$ (including empty block) is denoted by $\Lambda^*$. 

We say that a block $w=w_1\ldots w_n\in \Lambda^*$ \emph{occurs in $x=(x_i)_{i=1}^\infty\in\Lambda^\N$} and $x$ \emph{contains} $w$ if $w_j=x_{i+j-1}$ for some $i\in\N$ and all $1 \le j\le n$. The empty block occurs in every point of $\Lambda^\N$. Similarly, given an $n$-block $w=w_1\ldots w_n \in\Lambda^*$, a \emph{subblock} of $w$ is any block of the form $v=w_iw_{i+1}\ldots w_j\in\Lambda^*$ for each $1\le i\le j\le n$. A \emph{language} of a shift space $X\subset\Lambda^\N$ is the set $\Bl(X)$ of blocks over $\Lambda$ which occur in some $x\in X$. The language of the shift space determines it: two shift spaces are equal if and only if they have the same language \cite[Proposition 1.3.4]{LM95}. To define a shift space it is enough to specify a set $\mathcal{L}\subset\Lambda^*$ which is \emph{factorial}, meaning that if $u\in\mathcal{L}$ then so does any subblock of $u$, and \emph{prolongable}, meaning that for every block $u$ in $\mathcal{L}$ there is a symbol $a\in \Lambda$ such that the concatenation $ua$ also belongs to $\mathcal{L}$.

It is convenient to adapt definitions of the weak specification and almost specification property to symbolic dynamics.

We say a non-decreasing function $\theta \colon \Zp\to\Zp$ is a mistake function if $\theta(n) \le n$ for all $n$ and $\theta(n)/n \to 0$.
A shift space has the \emph{almost specification property} if there exists a mistake function $\theta$ such that for every $n\in\N$ and $w_1,\ldots,w_n\in\lang(X)$, there exist words $v_1,\ldots,v_n\in\lang(X)$ with $|v_i| = |w_i|$ such that $v_1v_2\ldots v_n\in\lang(X)$ and each
$v_i$ differs from $w_i$ in at most $\theta(|v_i|)$ places.

We say that a shift space $X$ has the \emph{weak specification property} if for every $n\in \N$ there exists $t=t(n)\in\N$ such that $t(n)/n\to 0$ as $n\to\infty$ and for any words $u,w \in \lang(X)$ there exists a word $v\in \lang(X)$ such that $x = uvw \in\lang(X)$ and $|v| = t$.

Given an infinite collection of words $\mathcal{L}$ over an alphabet $\Lambda$, the entropy of $\mathcal{L}$ is $h(\mathcal{L}) = \limsup_{n\to \infty}
\frac{1}{n} \log \#(\mathcal{L} \cap \Lambda^n).$

\section{Almost specification and measures of maximal entropy}

In this section we construct a family of shift spaces which contains
\begin{enumerate}
               \item A shift space with almost specification and multiple measures of maximal entropy.
               \item A shift space with weak specification and multiple measures of maximal entropy.
               \item A shift space with almost specification but without weak specification.
               \item Shift spaces $X$ and $Y$  satisfying
               \begin{enumerate}
               \item $Y$ is a factor of $X$,
                 \item their languages possess the Climenhaga-Thompson decomposition
               $\lang(X)=\CpX\cdot\G_X\cdot\CsX$ and $\lang(Y)=\CpY\cdot\G_Y\cdot\CsY$,
                 \item $h(\G_X)>h(\CpX\cup\CsX)$ and $h(\G_Y) < h(\CpY\cup\CsY)$,
                 \item $X$ is intrinsically ergodic, while $Y$ is not.
               \end{enumerate}
\end{enumerate}
As we want to kill two (actually, more than two) birds with one stone, therefore our construction is a little bit more involved than needed for each of our goals separately. We use the flexibility to shorten the total length of the paper.

\subsection{Construction of $\XR$} Our aim is to construct a shift space, denoted by $\XR$, for a given integers $p,q\in \N$ with $q\ge 2$ and a family of sets $\RR=\{R_n\}_{n=1}^\infty$. Needless to say, $\XR$ and its properties rely on these parameters. Our notation will not reflect the dependence
on $p$ and $q$.

\subsubsection{Parameters} Fix integers $p,q\in\N$, $p,q\ge 2$. Let $\RR=\{R_n\}_{n=1}^\infty$ be an increasing sequence of nonempty finite subsets of $\N$ such that
$\max R_n\le n$ for each $n\in\N$. That is,
\[
\{1\}=R_1\subset R_2\subset R_3\subset \ldots \text{ and }R_n\subset\{1,\ldots,n\}.
\]
One may think that the elements of $R_n$ are the special positions in a word of length $n$.

We define a nondecreasing function $r\colon\Zp\to\Zp$ by $r(0)=0$ and
\[
r(n)=|R_n| \quad\text{for }n\in\N.
\]
The function $r(n)$ can be interpreted as the count of the number of special positions in a word of length $n$.
We have $r(n-1)\le r(n)$ for each $n\in\N$.

We say that the set $R_n$ has \emph{a gap of length $k$} if $\{1,\ldots,n\}\setminus R_n$ contains $k$ consecutive integers. 
By $N_k$ we denote the smallest $n$ such that $R_n$ has a gap of length $k$ (if such an $n$ exists, otherwise we set $N_k=\infty$). We say that the set $\RR$ has \emph{large gaps} if $N_k<\infty$ for all $k\in\N$. It is easy to see that the monotonicity condition ($R_n\subset R_{n+1}$ for $n\in\N$) implies that some set $R_m$ in $\RR$ has a gap of length $k$ if and only if for each $k$ we have $k\ge n-\max R_{n}$ for infinitely many $n\in\N$. Note that $r(n)/n\to 0$ as $n\to\infty$ implies that $\{R_n\}_{n=1}^\infty$ has large gaps.

\subsubsection{Definition of $\XR$} Let $\alf=\{1,\ldots,p\}\times\{0,1,\ldots,q-1\}\cup\{(0,0)\}$. We will depict $(a,b)\in \alf$ as $\tword{a}{b}$ and regard $a\in\{0,1,\ldots,p\}$ as the \emph{color} of the whole symbol. We call the symbol $\tword{0}{0}$ the \emph{marker} symbol and denote the block of length one containing the marker by $\OO$.
We say that a word $\tword{a_1\ldots a_n}{b_1\ldots b_n}\in\alf^*$ is \emph{monochromatic} or \emph{of color} $a\in \{0,1,\ldots,p\}$ if $a=a_1=\ldots=a_n$, and \emph{polychromatic} otherwise.
We use capital letters to denote blocks (words) over $\alf$ to remind that they can be identified with matrices.
We say that a subblock
\[
V=\tword{a_i\ldots a_j}{b_i\ldots b_j}{}
\]
is a \emph{maximal monochromatic subword} of
\[
W=\tword{a_1\ldots a_n}{b_1\ldots b_n}\in\alf^*
\]
if $a_i=a_{i+1}=\ldots=a_j$ and $1\le i \le j\le n$ are such that $i=1$ or $a_{i-1}\neq a_i$, and $j=n$ or $a_j\neq a_{j+1}$.
Furthermore, if $i=1$ ($j=n$), then we say that $V$ is a \emph{maximal monochromatic prefix} (\emph{suffix}, respectively) of $W$.

We define the language of a shift by specifying which words are allowed.  We declare all monochromatic words of color $a\in\{1,\ldots,p\}$ to be allowed.
The blocks $\OO^k$ for $k=1,2,\ldots$ are the only allowed monochromatic blocks of color $0$. We denote the set of all monochromatic allowed words by $\Mon$.
We put some constraints on the polychromatic blocks. 
We say that a monochromatic block
\[
W=\tword{a_1\ldots a_n}{b_1\ldots b_n}\in\alf^*
\]
in color $a\in\{1,\ldots,p\}$ is \emph{restricted} if $b_{j}=0$ for each $j\in R_k$. In other words, some symbols in the second row of a restricted block are set to $0$. To describe where these $0$'s must appear we use the sequence $\RR=\{R_n\}_{n=1}^\infty$.
We write $\Rst$ for the set of all restricted words. We agree that empty block is both, restricted and monochromatic, block. We say that a block $W$ is \emph{free} if $W=\OO V$, where $V$ is a (possibly empty) restricted block.
Let $\F$ be the set of all free blocks. We call any member of a set $\G=\F^*$ of all finite concatenations of free blocks \emph{a good word}. A word is \emph{allowed} if it can be written as a concatenation of a monochromatic word and good word, that is, $W$ is allowed if there exist $U\in\Mon$ and $V_1,\ldots,V_k\in\F$ such that $W=UV_1\ldots V_k$.

It is easy to see that the set of allowed words is a language of a shift space, which we denote by $\XR$.

\subsection{Dynamics of $\XR$}

We recall that a shift space $X$ is \emph{synchronized} if there exists a \emph{synchronizing word} for $X$, that is, there is a word $v\in\Bl(X)$ such that $uv,vw\in\Bl(X)$, implies $uvw\in\Bl(X)$.

\begin{lemma}
The shift space $\XR$ is synchronized.
\end{lemma}
\begin{proof} It is easy to see that $\OO$ is a synchronizing word for $\XR$.
\end{proof}

\begin{proposition}
If $\RR$ has large gaps, then the shift space $\XR$ is topologically mixing with dense periodic points.
\end{proposition}
\begin{proof} 
The existence of large gaps implies that for each $k\in\N$ we can pick $N_k\in\N$ such that any monochromatic word of color $a\in\{1,\ldots,p\}$ and length at most $k$ is a suffix of some restricted word of length $N_k$. It follows that for each monochromatic word $V$ there exists a thick set $T\subset\N$ such that $V$ is a suffix of some free word of length $t$ for every $t\in T$. This easily implies that $\XR$ is weakly mixing. We conclude the proof by noting that every weakly mixing synchronized shift is mixing (e.g. see \cite[Prop. 4.8]{OM}) and has dense set of periodic points.
\end{proof}

\begin{lemma}\label{lem:almost_gdecreasing}
If $r(n)/n\to 0$ as $n\to\infty$, then the shift space $\XR$ has the almost specification property.
\end{lemma}
\begin{proof} 
We claim that the function $\theta$ given by $\theta(n)=r(n-1)+1$ is a mistake function for $\XR$.
It is enough to show that given any two $\XR$-allowed words
\[
U=\tword{a_1\ldots a_{k}}{b_1\ldots b_{k}}\quad\text{and}\quad V=\tword{c_1\ldots c_{l}}{d_1\ldots d_{l}}
\]
we can change $V$ in at most $\theta(l)=r(l-1)+1$ positions to find a word $V'$ which is free. Then the concatenation $UV'$ is an $\XR$-allowed word.
To define $V'$ we change $\tword{c_1}{d_1}$ to $\OO$ and modify the maximal monochromatic prefix of $\tword{c_2\ldots c_{l}}{d_2\ldots d_{l}}$ by putting at most $r(l-1)$ zeros on the restricted positions in the second row. Such $V'$ is clearly free and $UV'$ is then an $\XR$-allowed word.
\end{proof}

\begin{lemma}\label{lem:large_g_weak}
The shift space $\XR$ has the weak specification property if and only if $\{R_n\}_{n=1}^\infty$ has large gaps and $k/N_k\to 1$ as $k\to\infty$.
\end{lemma}

\begin{proof}
Assume that $\{R_n\}_{n=1}^\infty$ has large gaps and $k/N_k\to 1$ as $k\to\infty$.
Take any two allowed words
\[
U=\tword{a_1\ldots a_{j}}{b_1\ldots b_{j}}\quad\text{and}\quad W=\tword{c_1\ldots c_{k}}{d_1\ldots d_{k}}.
\]
Without loss of generality we may assume that $W$ is monochromatic and in color $a\in\{1,\ldots,p\}$.
We can find an allowed word $V$ of length $N_k-k$ such that the concatenation $VW$ is a restricted word. Therefore $U\OO^k VW$ is  an allowed word for every $k\geq 1$. Now $(N_k-k)/N_k\to 0$ as $k\to\infty$ implies that $\XR$ has the weak specification property.

Now assume that $\XR$ has the weak specification property. Pick a color $a\neq 0$ and let
$U=\OO$ and $W_k=\tword{a}{1}^k$ for $k\in\N$. Then $W_k$ is a monochromatic, but not restricted word of length $k$. Weak specification implies that for each $k\in\N$ there is a word $V$ such that $UVW_k$ is an allowed word. Let $V_k$ be the shortest such word. Note that $|V_k|\ge 1$. We know that $|V_k|/k\to 0$ as $k\to\infty$. By the definition of $\XR$ the maximal monochromatic suffix of $V_kW_k$ has to be a restricted word. Let $j(k)$ be its length. It is easy to see that $N_k\le j(k)$. We also have $k<j(k)\le |V_k|+k$. It is now clear that $j(k)-\max R_{j(k)}\ge k$ and $k/j(k)\to 1$ as $k\to\infty$, which implies $k/N_k\to 1$ as $k\to\infty$, and completes the proof.
\end{proof}

\subsection{Entropy of $\XR$}
In this section we collect some auxiliary estimates for entropy of $\XR$.
\begin{lemma}\label{lem:counting}
Let $\F_n$ and $\G_n$ denote the number of $n$-blocks in $\F$ and $\G$, respectively. Then
\begin{enumerate}
  \item \label{restricted} $\F_0=\F_1=1$ and $\F_n=p\cdot q^{(n-1)-r(n-1)}$ for all $n>1$;
  \item \label{good} $\G_0=\G_1=1$ and
  \begin{equation}\label{g-rec}
  \G_n=\sum_{i=1}^{n} \F_i\G_{n-i}=\G_{n-1}+\sum_{j=1}^{n-1} pq^{j-r(j)}\cdot\G_{n-1-j}\text{ for }n>1.
  \end{equation}
\end{enumerate}
\end{lemma}
\begin{proof}
The first point is obvious. The equalities $\G_0=\G_1=1$ follow from the definition of $\G$. Let $n\ge 2$ and let $W\in\G_n$ be a concatenation of free words. Then $W$ must end with a free word $V$ of length $j\in\{1,\ldots,n\}$ which has the form $V=\OO U$ for some restricted word $U$. The word $V$ can be chosen in $\F_{j}$ different ways, hence the formula.
\end{proof}

The following inequality
\begin{equation}
1+p\sum_{j=1}^\infty q^{-r(j)}\le q. \label{condition}
\end{equation}
is crucial for the uniqueness of a measure of maximal entropy for $\XR$.
We first note some conditions which should be imposed on $r(n)$ to guarantee that
$\eqref{condition}$ holds for some $p$ and $q$.

\begin{lemma}\label{lem12}
If $r(n)>0$ for every $n$ and
\[
\liminf_{n\to\infty}\frac{r(n)}{\ln n}>0,
\]
then there is $Q\geq 2$ such that the series
\[
\sum_{n=1}^\infty q^{-r(n)}
\]
converges for all integers $q\ge Q$ and its sum tends to $0$ as $q\to\infty$.
\end{lemma}
\begin{proof}
Observe that there is an integer $N>0$ and $c>0$ such that $r(n)>c \ln n= \frac{c}{\log_q e}\log_q n $ for all $n>N$.
Then
$$
\sum_{n=1}^\infty q^{-r(n)} \le \sum_{n=1}^N q^{-r(n)} + \sum_{n=N+1}^\infty \frac{1}{n^{\frac{c}{\log_q e}}},
$$
hence it is enough to take $Q>2$ so large that $c / \log_Q e>1$.
\end{proof}

By Lemma~\ref{lem12} given any function $r\colon\N\to \N$ such that $r(n)>0$ for every $n$ and
\[
\liminf_{n\to\infty}\frac{r(n)}{\ln n}>0,
\]
for any integer $p\geq 2$  we can find $q>p$ such that the inequality \eqref{condition} holds. Furthermore,
if
\[
\lim_{n\to\infty}\frac{r(n)}{\ln n}=\infty,
\]
then the series from the left hand side of \eqref{condition} converges for all $q\ge 2$.

\begin{lemma}\label{entropy1}
If \eqref{condition} holds, then $\G_n\le q^n$ for every $n\ge 0$.
\end{lemma}
\begin{proof}We use the induction on $n$. We have $\G_0=\G_1=1$.
Assume the assertion is true for $j=0,1,\ldots, n-1$ where $n\in\N$.
Using recurrence relation \eqref{g-rec} we have
\begin{align*}
\G_n&=\G_{n-1}+\sum_{i=1}^{n-1} pq^{i-r(i)}\cdot\G_{n-1-i}\\
    &\le q^{n-1}+\sum_{i=1}^{n-1} pq^{i-r(i)} \cdot q^{n-1-i} = q^{n-1}\cdot \big(1+p\cdot \sum_{j=1}^n q^{-r(j)}\big)\\
    &\le q^{n-1}\cdot \big(1+p\cdot \sum_{j=1}^\infty q^{-r(j)}\big)\le q^n.\qedhere
\end{align*}
\end{proof}

\begin{lemma}\label{entropy2}
If \eqref{condition} does not hold, then
\[
\liminf_{n\to\infty} \frac{\log \G_n}{n}>\log q.\]
\end{lemma}
\begin{proof}
Assume that \eqref{condition} does not hold, that is,
\[
1+p\sum_{j=1}^\infty q^{-r(j)}> q.
\]
Then we can find $N\in\N$ and $z>1$ such that
\begin{equation}\label{z-cond}
1+\sum_{j=1}^{N-1}
pq^{-r(j)}>qz^{N+1}.
\end{equation}
We claim that for all $n$ it holds
\begin{equation}
\G_n\ge (qz)^{n-N}. \label{ineq}
\end{equation}
It is clear that \eqref{ineq} is true for all $n\le N$. For the induction step we assume that
there for some $n\ge N$ the inequality \eqref{ineq} holds for all $0\le k\le n$ .

By \eqref{g-rec} we have that
\begin{align}
\notag \G_{n+1}&\ge \G_n+\sum^{N-1}_{j=1}\G_{n-j}pq^{j-r(j)}\\
\notag &\ge (qz)^{n-N}+\sum_{j=1}^{N-1} (qz)^{n-j-N}pq^{j-r(j)}\\
\label{almost}&= (qz)^{n-N}\bigg(1+\sum_{j=1}^{N-1} z^{-j}pq^{-r(j)}\bigg).
\end{align}
We have
$z^{-j}\ge z^{-N}$ hence
\[
\sum_{j=1}^{N-1} z^{-j}pq^{-r(j)}\ge z^{-N}\sum_{j=1}^{N-1} pq^{-r(j)}.
\]
Furthermore,
\begin{align*}
1+z^{-N}\sum_{j=1}^{N-1} pq^{-r(j)}&=z^{-N}\bigg(z^N+\sum_{j=1}^{N-1} pq^{-r(j)}\bigg)\\
&\ge z^{-N}\bigg(1+\sum_{j=1}^{N-1} pq^{-r(j)}\bigg)\\
&\ge qz.
\end{align*}
The last inequality above follows from \eqref{z-cond}. We have proved that
\begin{equation}\label{ineq2}
1+\sum_{j=1}^{N-1} z^{-j}pq^{-r(j)}\ge qz.
\end{equation}
We finish the proof by combining \eqref{ineq2} and \eqref{almost}.
\end{proof}


\begin{lemma}\label{entropy}
If \eqref{condition} holds, then $\htop(X)=\log q$.
\end{lemma}
\begin{proof}
Every allowed word of length $n\ge 2$ is either one of $1+pq^n$ monochromatic words or starts with a monochromatic, not necessarily restricted word of color $a\in\{1,\ldots,p\}$ and length $0\le i \le n-1$ followed by a concatenation of free words. 
Therefore
\[
|\Bl_n(X)|= 1+pq^n+\G_n+\sum_{i=1}^{n-1} (1+pq^i)\G_{n-i}.
\]
In particular, $|\Bl_n(X)|\ge q^n$.
It follows from Lemma \ref{entropy1} that $\G_n\le q^n$ for all $n$, hence
\[
|\Bl_n(X)|\le 1+pq^n+q^n+\sum_{i=1}^{n-1} (1+pq^i)q^{n-i}\leq 1+n(p+1)q^n.
\]
It is now clear that
\[
\htop(X)=\lim_{n\to\infty}\frac{1}{n}\log|\Bl_n(X)|= \log q.\qedhere
\]
\end{proof}

\begin{corollary}\label{cor:not-ie}
If \eqref{condition} holds, then $\XR$ has at least $p$ ergodic measures of maximal entropy.
\end{corollary}
\begin{proof}The set $X_a$ of all sequences with all symbols in the upper row in color $a\in\{1,\ldots,p\}$ is clearly an invariant subsystem with entropy $\log q$. The result follows by Lemma~\ref{entropy}.
\end{proof}

Note that the supports of measures of maximal entropy are nowhere dense and disjoint (provided that $p\ge 2$).

\subsection{Climenhaga-Thompson decompositions} We recall the notion introduced in \cite{CT}.
We say that the language $\lang(X)$ of a shift space $X$ \emph{admits Climenhaga-Thompson decomposition} if there are subsets
$\Cp$, $\G$, $\Cs$ satisfying following conditions:
\begin{enumerate}[(I)]
                 \item \label{cond:I} for every $w\in\lang (X)$ there are $u_p\in \Cp$, $v \in \G$,  $u_s \in\Cs$ such that $w=u_p v u_s$.
                 \item \label{cond:II} there exists $t\in\N$ such that for any $n \in \N$ and $w_1, \ldots,w_n \in \G$, there exist $v_1,  \ldots , v_{n-1}\in \lang(X)$ such that $x = w_1v_1w_2v_2 \ldots v_{n-1}w_n \in\lang(X)$ and $|v_i| = t$ for $i=1,\ldots,n-1$.
                 \item \label{cond:III} For every $M \in\N$, there exists $\tau$ such that given $w \in \lang(X)$
                 satisfying $w=u_p v u_s$ for some $u_p\in \Cp$, $v \in \G$,  $u_s \in\Cs$, with $|u_p| \le M$ and $|u_s| \le M$, there exist words $u',u''$ with $|u'| \le\tau$, $|u''| \le\tau$ for which $u'wu'' \in \G$.
\end{enumerate}


\begin{proposition}\label{prop:CT-decompose}
Let $\Cp=\Mon$, $\Cs=\emptyset$ and $\G$ be the collection of all good words. If the sequence $\{R_n\}_{n=1}^\infty$ has large gaps, then $\Cs,\G,\Cp$ is a Climenhaga-Thompson decomposition for $\XR$.
\end{proposition}
\begin{proof}The condition \eqref{cond:I} is a direct consequence of the definition of $\XR$.

To prove that \eqref{cond:II} holds with $t=0$ just note that the concatenation of any two good words
is again a good word.

For a proof of \eqref{cond:III} we fix $M\in\N$ and take any $W \in \lang(\XR)$ such that $W=U_p V$ for some $U_p\in \Cp=\M$ with $|U_p|=M$ and $V \in \G$.
If $U_p$ starts with $\OO$, then $W$ is already a good word and we can extend $W$ to another good word by adding as many symbols $\OO$ from the left as we want. If $U_p$ is of color $a\in\{1,\ldots,p\}$, then the length of the maximal monochromatic prefix of $W$ is equal to $|U_p|= M$. Hence we can extend $U_p$ to a restricted word by adding a prefix in the same color of length $N_{M+1}$. then adding $\OO$ as a prefix we obtain a free word which we can freely concatenate with $V$ to obtain a good word. Therefore \eqref{cond:III} holds with $\tau=N_{M+1}+1$.
\end{proof}

\begin{proposition}\label{prop:equivalence}
Let $\G$, $\Cp$ and $\Cs$ be as above. The condition \eqref{condition} does not hold if and only if
\[
h(\G)>h(\Cs\cup\Cp)=h(\Cp).
\]
\end{proposition}
\begin{proof}
Note that
\[
\htop(\Cs\cup\Cp)=\htop(\Mon)=\lim_{n\to\infty}\frac{1}{n}\log\Mon_n=\log q.
\]
If $\htop(\G)>\htop(\Cs\cup\Cp)=\log q$ then $\htop(X)>\log q$ and hence by Lemma \ref{entropy} condition \eqref{condition} does not hold.
The converse implication follows from Lemma~\ref{entropy2}.
\end{proof}
\begin{proposition}\label{prop:intrinsic-equiv}
The shift space $\XR$ is intrinsically ergodic if and only if the condition \eqref{condition} does not hold.
\end{proposition}
\begin{proof} If the condition \eqref{condition} does not hold, then Propositions \ref{prop:equivalence} and \ref{prop:CT-decompose} allow us to apply the Climenhaga-Thompson result (\cite[Theorem C]{CT}) and deduce that $\XR$ is intrinsically ergodic.  If condition \eqref{condition} holds, then Corollary \ref{cor:not-ie} implies that $\XR$ is not intrinsically ergodic.
\end{proof}

\subsection{Some concrete examples of $\RR$} So far, we have not proved that a suitable sequences $\{R_n\}_{n=1}^\infty$ exist.
We fill this gap and provide concrete examples of $\XR$.

\begin{rem}\label{rem:sum-formula}
Let $q\ge 2$. If we put $r(n)=\lfloor\sqrt{n}\rfloor$ for each $n\in\N$, then simple calculations
show that
\[
\sum_{n=1}^\infty q^{-r(n)}= \frac{3q-1}{(q-1)^2}.
\]
\end{rem}

\begin{proposition} \label{prop:almost-not-weak}
There exists a sequence $\{R_n\}_{n=1}^\infty$ and $q>p\geq 2$ such that $r(n)/n\to 0$ as $n\to\infty$,
condition \eqref{condition} is satisfied, and the shift space $\XR$ has the almost specification property, but it does not have the weak specification property.
\end{proposition}
\begin{proof}
If we denote
\[
R_n=[1,n]\cap\{k^2:k\in \N\},
\]
then $r(n)=|R_n|=\lfloor\sqrt{n}\rfloor$, hence $\XR$ has the almost specification property by Lemma~\ref{lem:almost_gdecreasing}.
By Remark~\ref{rem:sum-formula} we have
\[
\sum_{n=1}^\infty q^{-r(n)}=\frac{3q-1}{(q-1)^2},
\]
therefore taking $p= 2$  and $q=4$ we assure that \eqref{condition} holds. Next, observe that in $R_{k^2}$ the largest gap has length $k^2-(k-1)^2=2k-1$ hence $N_{2k}> k^2$, in particular
$\liminf_{k\to \infty} k/N_k=0$
and so the proof is finished by Lemma~\ref{lem:large_g_weak}.
\end{proof}

\begin{proposition} \label{prop:almost-weak}
There exists a sequence $\{R_n\}_{n=1}^\infty$ and $q>p\geq 2$ such that $r(n)/n\to 0$ as $n\to\infty$,
condition \eqref{condition} is satisfied and the shift space $\XR$ has the weak specification property.
\end{proposition}
\begin{proof}
If we denote
$$
R_n=\set{1,2,\ldots, \lfloor\sqrt{n}\rfloor}
$$
then $r(n)=|R_n|=\lfloor\sqrt{n}\rfloor$. By Remark~\ref{rem:sum-formula} we have
\[
\sum_{n=1}^\infty q^{-r(n)}=\frac{3q-1}{(q-1)^2}.
\]
Hence as before, taking $p= 2$  and $q=4$ we assure that \eqref{condition} holds.
Note that for every $\eps>0$ there is $K\in\N$ such that $\eps K>1$ and $\sqrt{k}<\eps k$ for every integer $k>K$.
Fix any $k>K$. To estimate $N_k$ we need to find minimal $m$ such that $k+\sqrt{m}\leq m$.
By the choice of $\eps$, this condition is clearly satisfied when $k\leq (1-\eps)m$.
In particular, $N_k\leq 1+k/(1-\eps)$ which implies that  $\lim_{k\to \infty} k/N_k\geq (1-\eps)$.
But $\eps$ can be arbitrarily small, hence $\lim_{k\to \infty} k/N_k=1$
and so the proof is finished by Lemma~\ref{lem:large_g_weak}.
\end{proof}

\subsection{The main theorem}
We are now in position to prove the main result of this section.

\begin{theorem}\label{colors:examples} The following assertions hold:
\begin{enumerate}[(i)]
               \item \label{cond:almost-multiple-mme} there is a shift space with the almost specification property and multiple measures of maximal entropy;
               \item \label{cond:weak-multiple-mme} there is a shift space with the weak specification property and multiple measures of maximal entropy;
               \item \label{cond:almost-not-weak} there is a shift space with the almost specification property but without the weak specification property;
               \item \label{cond:CT} there is a shift space $X$ and its shift factor $Y$ such that
               \begin{enumerate}
                 \item the languages of $X$ and $Y$ have the Climenhaga-Thompson decompositions
               $\lang(X)=\CpX\cdot\G_X\cdot\CsX$ and $\lang(Y)=\CpY\cdot\G_Y\cdot\CsY$, \label{cond:CT:a}
                 \item $h(\G_X)>h(\CpX\cup\CsX)$ and $h(\G_Y) < h(\CpY\cup\CsY)$,\label{cond:CT:b}
                 \item $X$ is intrinsically ergodic,\label{cond:CT:c}
                 \item $Y$ has multiple measures of maximal entropy.\label{cond:CT:d}
               \end{enumerate}
             \end{enumerate}
\end{theorem}

\begin{proof}
Conditions \eqref{cond:almost-multiple-mme} and \eqref{cond:weak-multiple-mme}
follow by Proposition~\ref{prop:almost-weak} and Corollary~\ref{cor:not-ie},
while \eqref{cond:almost-not-weak} follows by Proposition~\ref{prop:almost-not-weak} and Corollary~\ref{cor:not-ie}.

Note that
$
R_n=\set{1,2,\ldots, \lfloor\sqrt{n}\rfloor}
$
has large gaps. Let $\RR=\{R_n\}_{n=1}^\infty$.
Put $q=4$ and let $X$ be given by our construction for $p=3$ and $Y$ be given by our construction for $p=2$.
There is a factor map $\pi\colon X\to Y$ given by the $1$-block map given by $\tword{a}{b} \mapsto \tword{a}{b}$, if $a\in\{0,1,2\}$
and $b\in \set{0,1,2,3}$; and $\tword{3}{b} \mapsto \tword{2}{b}$. By Remark~\ref{rem:sum-formula}
condition \eqref{condition} is not satisfied when $p=3$ and is satisfied when $p=2$. By Proposition~\ref{prop:intrinsic-equiv} the shift space $X$
is intrinsically ergodic while $Y$ has at least two measures of maximal entropy by Corollary~\ref{cor:not-ie}. Proposition~\ref{prop:equivalence}
implies \eqref{cond:CT:a} and \eqref{cond:CT:b}.
\end{proof}

\begin{rem} Climenhaga and Thompson proved that if a shift space $X$ has the decomposition named after them given by the sets $\Cp$, $\G$, $\Cs$, and
$h(\G)>h(\Cp\cup \Cs)$, then $X$ is intrinsically ergodic (\cite[Theorem C]{CT}). Furthermore, if $h(\Cp\cup \Cs)=0$, then every shift factor of $X$ is intrinsically ergodic as well (\cite[Theorem D]{CT}). Shift space constructed in Theorem~\ref{colors:examples}\eqref{cond:CT} shows that Theorem~D in \cite{CT} is the best possible --- it does not hold when $h(\CpX\cup\CsX)>0$.
\end{rem}

\begin{rem} Invariant measures of $\XR$ and their entropy can be analyzed by methods introduced by Thomsen \cite{Thomsen}. Thomsen's theory applies because $\XR$ is synchronized.
The following claims can be proved in a straightforward way:
\begin{enumerate}
\item the Markov boundary of $\XR$ is a disjoint union of subsystems $X_a$ of monochromatic sequences in a given color $a\in\{1,\ldots,p\}$,
\item the entropy of the Fisher cover of $\XR$ is equal to $h(\G)\ge\log q$.
\end{enumerate}
It follows that $\XR$ has either 
$p+1$ ergodic measures of maximal entropy if $h(\G)=\log q$, or a unique, fully supported measure of maximal entropy if $h(\G)>\log q$. We refer the reader to \cite{Thomsen} for the definitions of Markov boundary and Fisher cover.
\end{rem}


\section{Almost specification, u.p.e. and horseshoes}

The notion of \emph{uniform positive entropy (u.p.e.)} dynamical systems
was introduced in \cite{Blanchard}, as an analogue in topological
dynamics of the notion of a $K$-process in ergodic theory. In particular, every
non-trivial factor of a u.p.e. system has positive topological entropy.
Recall that an open cover $\U = \{U_1, \ldots,U_m\}$ of $X$ is called \emph{standard}
if every $U_j\in\U$ is non-dense in $X$.
A dynamical system $(X, T)$ has \emph{uniform positive
entropy (u.p.e.)} if for every standard cover $\U = \{U, V\}$ of $X$, the topological entropy
$h(T,\U)$ is positive. A pair $(x, x') \in X \times X$ is an \emph{entropy pair} if for every standard
cover $\U = \{U, V\}$ with $x\in \Int(X\setminus U)$ and  $x'\in \Int(X\setminus V)$ we have $h(T,\U) > 0$. Equivalently, $(X,T)$ is
u.p.e. if every nondiagonal pair in $X \times X$ is an entropy pair.
Generalizing the notion of an entropy pair, Glasner and Weiss \cite{GW} call
an $n$-tuple $(x_1,x_2, \ldots , x_n) \in X \times\ldots\times X$ an \emph{entropy $n$-tuple} if
at least two of the points $\{x_j\}^n_{j=1}$ are different and whenever $U_j$ are
closed mutually disjoint neighborhoods of the distinct points $x_j$, the open
cover $\U =\{ X\setminus U_j : 0 < j \le n\}$ satisfies $h(T,\U) > 0$. We say that a system $(X,T)$ is \emph{topological K} if every non-diagonal tuple is an entropy tuple. We say that an open set $U\subset X$ is \emph{universally null} for $T$ if $\mu(U)=0$ for every $T$-invariant measure $\mu$.
The \emph{measure center} of a dynamical system $(X,T)$ is the complement of the union of all universally null sets.




\subsection*{Standing assumption} In the remainder of this section we assume that $(X,T)$ is a dynamical system with the almost
specification property, $g$ is a mistake function and $k_g$ corresponds to $g$.




The following result is proved implicitly in the proof of \cite[Theorem 6.8]{WOC}. For the reader's convenience we provide it with a proof.

\begin{theorem}\label{thm:min-tracing}
For every $m\geq 1$, $\eps_1,\ldots,\eps_{m} > 0$, $x_1, \ldots, x_{m} \in X$, and integers
$l_0=0< l_1< \ldots< l_{m-1}< l_m=L$ with $l_{j}-l_{j-1}\ge k_g(\eps_j)$ for $j=1,\ldots,m$
there is a minimal point $q\in X$ which $(g;l_j-l_{j-1},\eps_j)$ traces $T^{[l_{j-1},l_j)}(x_j)$ over
$[l_{j-1}+sL,l_j+sL)$
for every $j=1,\ldots,m$ and $s\in\Zp$.
\end{theorem}
\begin{proof}Fix $m\geq 1$, $\eps_1,\ldots,\eps_{m} > 0$, $x_1, \ldots, x_{m} \in X$, and integers
$l_0=0< l_1< \ldots< l_{m-1}< l_m=L$ with $n_j=l_{j}-l_{j-1}\ge k_g(\eps_j)$ for $j=1,\ldots,m$.
By the almost specification property closed sets
\[
C=\bigcap_{j=1}^m T^{-l_{j-1}}(B_{n_j}(g;x_j,\eps_j),\text{ and }C_s=\bigcap_{j=0}^s T^{-jL}(C)
\]
are nonempty. The set
\[
Z=\bigcap_{s=0}^\infty C_s
\]
is closed and nonempty because it is an intersection of a decreasing family of closed nonempty subsets. Moreover,
$Z$ is $T^L$ invariant, hence it contains a $T^L$-minimal point $q$. But $q$ must be then also minimal for $T$.
\end{proof}

It follows easily from Theorem \ref{thm:min-tracing} that the minimal points are dense in the measure center of a dynamical system with the almost specification property. For a proof see \cite[Theorem 6.8]{WOC}. 
The above result suggests that a system with the almost specification should have a lot of minimal subsystems. However, there are examples of proximal dynamical systems with the almost specification property and a unique minimal point, which is then necessarily fixed (see \cite{KKO}). The measure center of theses examples is trivial (it is the singleton of that fixed point) and hence they all have topological entropy zero.

Recall that we say that a dynamical system $(X,T)$ \emph{has a horseshoe} if there are an integer $K>0$ and a closed, $T^K$-invariant set $Z$ such that the full shift over a finite alphabet is a factor of $(Z, T^k|_Z)$.
\begin{theorem}
A dynamical system $(X,T)$ with the almost specification property restricted to the measure center is topological K. 
If the measure center is non-trivial, then $(X,T)$ has a horseshoe.
\end{theorem}
\begin{proof}
By \cite{WOC} if $(X,T)$ has the almost specification property then so does its restriction to the measure center.
Hence, without loss of generality we may assume that $(X,T)$ admits a fully supported measure.
Fix $m\ge 1$ and let $U_1,\ldots, U_m$  be nonempty open sets and let $\delta>0$  and $W_j\subset U_j$ for $j=1,\ldots,m$
be nonempty open sets such that $V^{\delta}_j=\{y\in X: \rho(x,y)\le \delta \text { for some }x\in \overline{W}_j\}\subset U_j$ for $j=1,\ldots,m$. Without loss of generality we may assume that $V^{\delta}_1,\ldots,V^{\delta}_m$ are pairwise disjoint.
Since $W_j$ is an open set, for each $j=1,\ldots,m$ there is an invariant measure
$\nu_j$ such that $\nu_j(W_j)>0$.  Let $\mu^{\star} = \nu _1\times \ldots \times \nu_m$ be the product measure on $X^m$.
By ergodic decomposition theorem there is an ergodic measure $\nu$ on $X^m$ such that $\nu(W_1\times\ldots\times W_m)>\eps>0$.
Let $(x_1,\ldots,x_m)\in X^m$ be a generic point for $\nu$.
For $j=1,\ldots,m$ and $k\in\N$ define $N_k(x_j,W_j)=\{0\le l<k: T^l(x_j)\in W_j\}$. Furthermore, let
\[
J_k=\{0\le l<k: T^l(x_j)\in W_j\text{ for any }j=1,\ldots,m\}=\bigcap_{j=1}^m N_k(x_j,W_j).
 \]
By ergodic theorem there exists $K\in\N$ such that for $j=1,\ldots,m$ and $k\ge K$ we have
$|J_k|\ge k\eps$. Take $k>K$ such that $k\ge k_g(\delta)$ and
$g(k,\delta)<k\eps/m$. 
Note that for every $A_1,\ldots, A_m\in I(g;k,\delta)$ we have
$$
|\bigcap_{i=1}^m A_i|\geq k - m g(k,\delta)> k(1-\eps).
$$
Therefore $J_k\cap A_1\cap\ldots\cap A_m\neq\emptyset$.

It follows that for any $s,t\in\{1,\ldots,m\}$ with $s\neq t$ we have
$B_k(g;x_s,\delta)$ and $B_k(g;x_t,\delta)$ are disjoint. Set $C_j=B_k(g;x_s,\delta)$.
Define
\[
Z=\bigcap_{s\in\Zp} T^{-sk}(C_1\cup\ldots\cup C_m).
\]
We have $T^k(Z)\subset Z$. Given $\xi\colon I\to\{1,\ldots,m\}$, where $I=\{0,1,\ldots,n-1\}$ or $I=\Zp$ let
\begin{equation}
Z_\xi=\bigcap_{s\in I} T^{-sk}(C_{\xi(s)}).\label{z:ksi}
\end{equation}
Clearly,
\[
Z=\bigcup_{\xi\in\set{1,\ldots,m}^\Zp} Z_\xi.
\]
It is easy to see that $Z$ and $Z_\xi$ are always closed and nonempty by the the almost specification property. Observe also that $z\in Z$ if and only if there is some $\xi\in \set{1,\ldots,m}^\Zp$ such that
$z\in Z_\xi$. Moreover, for any $I$ as above we have
$Z_{\xi'}\neq Z_{\xi''}$ provided $\xi',\xi''\colon I\to \set{1,\ldots,m}$ and $\xi'\neq\xi''$.
Therefore $\pi\colon Z\to\set{1,\ldots,m}^\Zp$ given by $\pi(z)=\xi$, where $\xi\in\set{1,\ldots,m}^\Zp$
is such that $z\in Z_\xi$ is a well-defined surjection. To see that $\pi$ is continuous note that for every word $w\in\set{1,\ldots,m}^{\set{0,1,\ldots,n-1}}$ we have $\pi^{-1}(C[w])=Z_w$ is closed, where $C[w]$ denotes the cylinder of $w$. Moreover, $Z_w$ is open as
\[
Z_w=Z\setminus \bigcup \set{Z_{w'}:w'\in\set{1,\ldots,m}^{\set{0,1,\ldots,n-1}},\, w'\neq w }
\]
has closed complement. It is now easy to see that $\pi$ is a factor map form a $T^k$ invariant set $Z$ onto the full shift on $m$ symbols. In other words, $(X,T)$ has a horseshoe.

It remains to prove that $(X,T)$ is topological $K$.
To this end, assume that the open sets $U_1,\ldots, U_m$ have pairwise disjoint closures.
We need to show that the cover $\U=\set{X\setminus\overline{U}_1,\ldots,X\setminus\overline{U}_m}$
has positive entropy. 
Assume on the contrary that $h(T,\U)=0$.
Therefore there is $n$ such that
$$
N\big(\bigvee_{i=0}^{kn-1}T^{-i}\U\big)<\left(\frac{m}{m-1}\right)^n
$$
as otherwise
$h(T,\U)\geq \log(\frac{m}{m-1})/k>0$.
Let $\V$ be a subcover of $\bigvee_{i=0}^{kn-1}T^{-i}\U$ with less than ${m}^n/({m-1})^n$ elements. For
each $\xi\in\{1,\ldots,m\}^{\{0,1,\ldots,n-1\}}$ fix a point $z_\xi\in Z_\xi$ (see \eqref{z:ksi}) and recall that the set  $\{z_\xi\}$ has exactly $m^n$ elements. Since $\V$ is a cover of $X$ there is
$V\in \V$ such that the set
$A=\set{\xi : z_\xi\in V}\subset \set{1,\ldots,m}^{\{0,1,\ldots,n-1\}}$ satisfies $|A|>(m-1)^n$.
It follows that there are $\xi^{(1)},\ldots,\xi^{(m)}\in A$ and $0\leq j<n$ such that $\left|\set{\xi^{(s)}_j : 1\leq s \leq m}\right|=m$.
Then there is $i\in J_k$ such that $T^{jk+i}(z_{\xi^{(s)}})\in U_{\xi^{(s)}_j}$.
By the definition of $V$ there exists $U\in \U$ such that $V\subset T^{-jk-i}(U)$, which means
that $T^{jk+i}(z_{\xi^{(s)}})\in U$ for each $s=1,\ldots,m$. But $U=X\setminus\overline{U}_s$ for some $1\leq s\leq m$ and there is also $r$ such that  $\xi^{(r)}_j=s$ which is a contradiction.
Hence $h(T,\U)>0$ 
and the proof is completed.
\end{proof}



\section{Weak specification does not imply almost specification}
We are going to construct a dynamical system with the periodic weak specification property, for which the almost specification property fails.
For every integer $m\ge 0$ let $X_m$ be the shift space over the alphabet $\set{a,b,c}$
given by the
following set of forbidden words:
\[
\mathcal{F}=\set{bc,cb}\cup \set{xa^k y^l : x,y\in \set{b,c},\, l\geq 1,\, 1\leq k \leq 2^m\lceil\log_2(l+1)\rceil}.
\]
Roughly speaking, the words allowed in $X_m$ consist of runs of $a$'s, $b$'s or $c$'s subject to the condition on the length of the run of $a$'s separating runs of $b$'s or $c$'s. Note that if $u,w$ are words allowed in $X_m$, then $ua^lw$ is also allowed provided that $l\ge 2^m(\lceil\log_2|w|\rceil+1)$. This shows that $X_m$ has the weak specification property.

Let $\mathbb{X}=\prod_{m=0}^\infty \set{a,b,c}^\N$. It is customary to think of elements of $\mathbb{X}$
as of infinite matrices from $\set{a,b,c}^{\N\times\N}$.
Hence we will use the matrix notation to denote points $\mathbf{x}$ in $\prod_{m=0}^\infty X_m$. We write $\mathbf{x}_{i\star}$ for the $i$-th row of $\mathbf{x}$, $\mathbf{x}_{\star j}$ for the $j$-th column and $\mathbf{x}_{ij}$ for the symbol in the row $i$ and column $j$. We endow $\mathbb{X}$
with the metric
$\rho(\mathbf{x},\mathbf{y})=\sup_{i=0,1,\ldots} 2^{-i} d(\mathbf{x}_{i\star},\mathbf{y}_{i\star})$.
It follows that for points $\mathbf{x},\mathbf{y}\in \mathbb{X}$
we have
\[
\rho (\mathbf{x},\mathbf{y})< 2^{-n} \text{ if and only if } \mathbf{x}_{ij}=\mathbf{y}_{ij} \text{ for }i+j<n.
\]
In other words, in order to $\mathbf{x}$ be close to $\mathbf{y}$ with respect to  $\rho$ the elements in the big upper left corners of matrices $\mathbf{x},\mathbf{y}$
must agree.
We define $S\colon \mathbb{X}\to \mathbb{X}$ to be the left shift, which acts on $\mathbf{x}\in \mathbb{X}$ by removing the first column and shifting the remaining columns one position to the left. It is easy to see that $S$ is continuous on  $\mathbb{X}$


Denote by $\mathbf{X}$ the subset of $\prod_{m=0}^\infty X_m\subset\mathbb{X}$ consisting of points $\mathbf{x}$ constructed by the following inductive procedure:
\begin{enumerate}
\item Pick $\mathbf{x}_{0\star}\in X_0$.
\item Assume that $\mathbf{x}_{i\star}\in X_0$ are given for $i=0,1,\ldots,m-1$ for some $m>0$. Pick any $\mathbf{x}_{m\star}\in X_m$ fulfilling
\[
\mathbf{x}_{mj}\in \set{b, \mathbf{x}_{(m-1)j}}\text{ for every }j.
\]
\end{enumerate}

Roughly speaking, when rows $0,1,\ldots,m-1$ are defined, we pick row $m$ so
that $\mathbf{x}_{m\star}$ is in $X_m$ and for each column we either rewrite a symbol from the same column in the row above, or we write $b$. Note that it means that $b$'s are persistent. In other words if $\mathbf{x}_{(m-1)j}=b$ for some $m$ and $j$ then we have to fill the rest of the column $j$ with $b$'s, that is $\mathbf{x}_{ij}=b$ for all $i\geq m$.
Clearly $\mathbf{X}$ is nonempty, closed and $S$-invariant 
($S(\mathbf{X})=\mathbf{X}$).

We will prove first that $(\mathbf{X},S)$ has the weak specification property. Given any $\eps>0$ pick $n\in\N$ such that $2^{-n}<\eps$ and let $m=2^n$. We claim that $M_\eps\colon\N\to\N$ given by
\[
M_\eps(l)=2^m\big(\lceil \log_2 (l+m)\rceil+1\big)+m
\]
is an $\eps$-gap function for $S$ on $\mathbf{X}$. 
Clearly, $M_\eps(l)/l\to 0$ as $l\to\infty$. Fix $K\in \N$, $\mathbf{x}^{(1)},\ldots,\mathbf{x}^{(K)}\in \mathbf{X}$ and integers $0=\beta_0<\alpha_1\le\beta_1<\alpha_2\le \beta_2 <\ldots \alpha_K\le\beta_K$. Assume that $\alpha_k-\beta_{k-1}\ge M_\eps(\beta_k-\alpha_k)$ for $k=1,2,\ldots,K$. Let $r=\alpha_1+\beta_K+m$. Define $\mathbf{x}$ by
\[
\mathbf{x}_{ij}=\begin{cases}
\mathbf{x}^{(k)}_{ij}, &\text{if }i<m \text{ and }(j\bmod r)\in [\alpha_k,\beta_k+m]\text{ for some }k\in\set{1,\ldots,K},\\
a, &\text{if }i<m \text{ and }(j\bmod r)\notin [\alpha_k,\beta_k+m]\text{ for all }k\in\set{1,\ldots,K},\\
b,&\text{if }i\ge m.
\end{cases}
\]
It is easy to see that $\mathbf{x}\in\mathbf{X}$ and a simple computation shows that
\[
\rho(S^j(\mathbf{x}),S^j(\mathbf{x}^{(k)}))<\eps
\]
for each $1\le k\le K$ and $\alpha_k\le j \le \beta_k$.
\begin{figure}
\[
\begin{array}{|ccc|ccc|ccc|}
\hline
\mathbf{x}^{(k)}_{0\alpha_k}&\ldots&\mathbf{x}^{(k)}_{0(\beta_{k}+m)}&a&\ldots&a&\mathbf{x}^{(k+1)}_{0\alpha_{k+1}}&\ldots&\mathbf{x}^{(k+1)}_{0(\beta_{k+1}+m)}\\
\vdots&\ddots&\vdots&\vdots&\ddots&\vdots&\vdots&\ddots&\vdots\\
\mathbf{x}^{(k)}_{(m-1)\alpha_k}&\ldots&\mathbf{x}^{(k)}_{(m-1)(\beta_k+m)}&a&\ldots&a&\mathbf{x}^{(k+1)}_{(m-1)0\alpha_{k+1}}&\ldots&\mathbf{x}^{(k+1)}_{(m-1)(\beta_{k+1}+m)}\\
\hline
b&\ldots&b&b&\ldots&b&b&\ldots&b\\
\vdots&&\vdots&\vdots&&\vdots&\vdots&&\vdots
\end{array}
\]
\end{figure}


We will show that $S\colon\mathbf{X}\to\mathbf{X}$ does not have the almost specification property.
Let $g\colon \N\times [0,\eps_0)\to\N$ be any function such that for every $0<\eps<\eps_0$ it holds $\lim_{l\to\infty}g(\eps,l)/l=0$. We are going to show that $g$ cannot be a mistake function for $S$. Assume on the contrary that $S$ has the almost specification property with the mistake function $g$.
First we set up some constants. 
Let $n\in\N$ be such that $0<2^{-n}<\eps_0$ and fix $\eps=2^{-n}$.
Pick $N>k_g(\eps)$ hence $g(N,\eps)<N$. Note that by the definition of $\rho$, our choice of $\eps$ and $N$ implies that for any point $\mathbf{y}$ which $(g,N,\eps)$-traces the orbit segment $S^{[a,b]}(\mathbf{x})$ of length $N$ 
there is $a\le j <b$ such that
$\mathbf{x}_{0j}=\mathbf{y}_{0j}$. Take any $M\in\N$ such that $2^M>2N$ and let $m\ge k_g(2^{-M})$.
Note that every point $\mathbf{y}$, which is $(g;m,2^{-M})$-tracing an orbit segment $S^{[0,m)}(\mathbf{x})$ we can find some $0\le j <m$ such that $\mathbf{x}_{Mj}=\mathbf{y}_{Mj}$. Let $s>2$ be an integer such that
\[
2^M\big(\lceil \log_2 ((s-2)\cdot N+1)\rceil\big)> N+m.
\]
Let $\mathbf{x}^{(0)}=c^{\N\times\N}$ and $\mathbf{x}^{(1)}=\ldots=\mathbf{x}^{(s)}=b^{\N\times\N}$.
Let $n_0=m$ and $n_j=m+j\cdot N$ for $j=1,\ldots,s$. 
Define a specification
\[
\xi=\{
S^{[0,n_0)}(\mathbf{x}^{(0)})
\}\cup\{
S^{[n_{j-1},n_j)}(\mathbf{x}^{(j)}) : \text{ for }j=1,2,\ldots,s
\}.
\]
By our
choice of parameters there is a point $\mathbf{y}$ which
$(g;2^{-M},m)$-traces $\xi$ 
over $[0,n_0)$ and $(g;N,\eps)$-traces $\xi$ 
over $[n_{j-1},n_j)$ for every $j=1,\ldots,s$. It follows that for each $j=1,\ldots,s$ there is $p_j\in[n_{j-1},n_j)$ such that $\mathbf{y}_{0p_j}=b$. Therefore in the $M$-th row of $\mathbf{y}$ there are symbols $b$ in columns $p_1,\ldots,p_s$. Note that $p_{j+1}-p_j<2N$ for $j=1,2,\ldots,s-1$. Hence by $2^M>2N$ together with the definition of $\mathbf{X}$ imply that $\mathbf{y}_{Mt}=b$ for every $p_1\le t \le p_s$, where $n_0\le p_1< n_1$ and $n_{s-1} \le p_s<n_s$. In particular, in the $M$-th row of $\mathbf{y}$ there are at least $(s-2)\cdot N$ consecutive $b$'s (all symbols between columns $n_2$ and $n_{s-1}$ inclusively). It follows that all symbols in $M$-th row of $\mathbf{y}$ in columns from $0$ to $m-1=n_1-1$ are either $a$'s or $b$'s, i.e. symbol $c$ cannot appear there. We have reached a contradiction since $\mathbf{y}$ is a point $(g;2^{-M},m)$-tracing $\mathbf{x}^{(0)}$  over $[0,m)$ which implies $\mathbf{x}_{Mj}=\mathbf{y}_{Mj}=c$ for some $0\le j <m$.

\section*{Acknowledgements}

We would like to thank Dan Thompson, Vaughn Climenhaga and Ronnie Pavlov for the numerous discussions and observations made on intrinsic ergodicity and the  specification property.

The research of Dominik Kwietniak was supported by the National Science Centre (NCN)
under grant no. DEC-2012/07/E/ST1/00185.




This project was initiated in June 2014 during Activity on Dynamics \& Numbers in the Max Planck Institute of Mathematics, Bonn.
It was continued during the visit of the third author (MR) to the Jagiellonian University in Krakow in July 2014.
The hospitality of these institutions is gratefully acknowledged.

\end{document}